\documentclass{amsart}

\usepackage{amsfonts,amssymb,amscd,amsmath,enumerate,verbatim,calc}
\usepackage{euscript}
\usepackage[all]{xy}

\newcommand{\CM}{Cohen-Macaulay}

\newcommand{\eF}{\EuScript{F}}

\newcommand{\wrt}{with respect to}

\newcommand{\n}{\mathfrak{n} }
\newcommand{\m}{\mathfrak{m} }
\newcommand{\tf}{\mathfrak{t} }
\newcommand{\M}{\mathfrak{M} }
\newcommand{\N}{\mathfrak{N} }
\newcommand{\q}{\mathfrak{q} }

\newcommand{\A}{\mathfrak{a} }
\newcommand{\B}{\mathfrak{b} }
\newcommand{\C}{\mathfrak{c} }
\newcommand{\R}{\mathcal{R} }

\newcommand{\Z}{\mathbb{Z} }

\newcommand{\nZ}{n \in \mathbb{Z} }

\newcommand{\Ee}{\EuScript{E}}

\newcommand{\GA}{G_{\mathfrak{a}}(A) }

\newcommand{\ra}{\EuScript{R}_{\mathfrak{a}}(A) }
\newcommand{\GF}{G_{\EuScript{F}}(A) }

\newcommand{\tF}{\EuScript{S}_{\EuScript{F}}(A)}
\newcommand{\ral}{\EuScript{R}_{\mathfrak{a}^l}(A) }

\newcommand{\ta}{\EuScript{S}_{\mathfrak{a}}(A) }

\newcommand{\eR}{\EuScript{R}}
\newcommand{\xb}{\mathbf{x}}

\newcommand{\eG}{\EuScript{G}}
\newcommand{\eE}{\EuScript{E}}
\newcommand{\eH}{\EuScript{H}}
\newcommand{\Sc}{\mathcal{S} }
\newcommand{\rt}{\rightarrow}
\newcommand{\xar}{\longrightarrow}
\newcommand{\ov}{\overline}

\newcommand{\Bcal}{\mathcal{B} }

\newcommand{\wh}{\widehat }

\newcommand{\image}{\operatorname{image}}

\newcommand{\grade}{\operatorname{grade}}
\newcommand{\depth}{\operatorname{depth}}
\newcommand{\ann}{\operatorname{ann}}
\newcommand{\red}{\operatorname{red}}
\newcommand{\htt}{\operatorname{ht}}

\newcommand{\embdim}{\operatorname{embdim}}

\newcommand{\Hs}{\operatorname{\ ^* Hom}}

\newcommand{\Hom}{\operatorname{Hom}}

\newcommand{\Ext}{\operatorname{Ext}}

\theoremstyle{plain}

\newtheorem{theorem}{Theorem}[section]

\newtheorem{corollary}[theorem]{Corollary}
\newtheorem{lemma}[theorem]{Lemma}

\theoremstyle{definition}
\newtheorem{remark}[theorem]{Remark}

\newtheorem{s}[theorem]{}

\theoremstyle{remark}

\numberwithin{equation}{theorem}

\begin{document}

\title[Itoh's conjecture]{A solution to Itoh's conjecture for integral closure filtration }
\author{Tony~J.~Puthenpurakal}
\date{\today}
\address{Department of Mathematics, Indian Institute of Technology, Bombay, Powai, Mumbai 400 076, India}

\email{tputhen@math.iitb.ac.in}
  \begin{abstract}
Let $(A,\m)$ be an analytically unramified  \CM \ local ring of dimension $d \geq 3$ and let $\A$ be an $\m$-primary ideal in $A$. If $I$ is an ideal in $A$ then let $I^*$ be the integral closure of $I$ in $A$. Let $\GA^*  = \bigoplus_{n\geq 0 }(\A^n)^*/(\A^{n+1})^*$ be the associated graded ring of the integral closure filtration of $\A$. Itoh conjectured  in 1992 that if third Hilbert coefficient of $\GA^*$, i.e.,  $e_3^{\A^*}(A) = 0$ and $A$ is Gorenstein then $\GA^*$ is \CM. In this paper we prove  Itoh's conjecture (more generally for analytically unramified \CM \ local rings).
  \end{abstract}
 \maketitle

\section{Introduction}
Let $(A,\m)$ be a Noetherian    local ring of dimension $d$ and let $\A$ be  an $\m$-primary ideal in $A$. We consider multiplicative $\A$-stable filtrations i.e.,
$\eF = \{ \A_n \}_{n\geq 0}$ such that
\begin{enumerate}[\rm (i)]
  \item $\A_n$ are ideals in $A$ for all $n \geq 0$ with  $\A_0 = A$.
  \item $\A \subseteq \A_1$ and $\A_1 \neq A$.
  \item $\A_n \A_m \subseteq \A_{n + m}$ for all $n, m \geq 0$.
  \item $\A \A_n = \A_{n + 1}$ for all $n \gg 0$.
\end{enumerate}
Let $\ta = \bigoplus_{n\geq 0} \A^n$ be the Rees ring of $A$ \wrt \ $\A$. Let $\tF =  \bigoplus_{n\geq 0} \A_n$ be the Rees ring of $A$ \wrt \ $\eF$. We have an inclusion of rings $\ta \hookrightarrow \tF$ and $\tF$ is a finite $\ta$-module.

If $I$ is an ideal in $A$ then set $I^*$ to be the integral closure of $A$. If $\eF = \{ \A_n \}_{n \geq 0}$ is a multiplicative $\A$-stable filtration on $A$ then note that $\A_n \subseteq \A_n^*$ for all $n \geq 1$.

\s \emph{Some examples of $\A$-stable filtrations:}\\
1. The $\A$-adic filtration, i.e., $\A_n = \A^n$ for all $n \geq 1$ is clearly an multiplicative $\A$-stable filtration. \\
2. Let $A$ be analytically unramified (i.e., the completion $\wh{A}$ is reduced). Then the integral closure filtration with respect to $\A$, i.e., $\A_n = (\A^n)^*$ is  a multiplicative $\A$-stable filtration, see \cite{Rees} (also see \cite[9.1.2]{Hun-Sw}). \\
3. Let $A$ be analytically unramified. By an asymptotic integral closure filtration  with respect to $\A$,  we mean a multiplicative $\A$-stable filtration $\eF = \{ \A_n \}_{n \geq 0}$ such that $\A_n = (\A^n)^*$ for all $n \gg 0$. Asymptotic integral closure filtration behave much better than the integral closure filtration when we go mod a superfical element (after going to a general extension of $A$), see
\ref{AtoA'}(d).

Let $\GA = \bigoplus_{n \geq 0} \A^n/\A^{n+1}$ be the associated graded ring of $A$ \wrt \ $\A$. Let $\eF = \{ \A_n \}_{n \geq 0}$ be a multiplicative $\A$-stable filtration. Consider
$\GF = \bigoplus_{n \geq 0} \A_n/\A_{n+1}$ be the associated graded ring of $A$ \wrt \ $\eF$. Let $\ell(M)$ denote the length of an $A$-module $M$.  Clearly $\GF $ is a finite $\GA$-module. It follows the function
$H_{\eF}(n) = \ell(A/\A_{n +1})$  is of polynomial type, i.e., there exists a polynomial $P_\eF(X) \in \mathbb{Q}[X]$ of degree $d$ such that $P_{\eF}(n) = H_{\eF}(n)$ for all $n \gg 0$. We write
$$P_\eF(X) = e_0^\eF(A)\binom{X+d}{d} -  e_1^\eF(A)\binom{X+d -1}{d -1} + \cdots + (-1)^d e_d^\eF(A), $$
where $e_i^\eF(A)$ are integers.
If $\eF$ is the $\A$-adic filtration then we set $e_i^\A(-)= e^\eF_i(-)$. Also if $\eF$ is the $\A$-integral closure filtration then we set ${e_i^\A}^*(-)= e^\eF_i(-)$.
Clearly $e_0^\eF(A) = e_0^\A(A) > 0$, for instance see \cite[11.4]{AM}.

\emph{For the rest of the introduction assume $A$ is \CM.}\\
Hilbert coefficients of \CM \ local rings satisfy many inequalities. It was J.Sally who first investigated when the border values are reached. She showed that  then $\GA$ has good properties (she mostly studied the case when
$\A = \m$). For instance Abhyankar proved that
$e_0^\m \geq \embdim A - \dim A + 1$, see \cite{A}.
Sally showed that if equality is attained (the so-called rings of minimal multiplicity) then $G_\m(A)$ is \CM, see \cite{Sa}. Northcott showed that $e_1^\A(A) \geq 0$, see \cite[Theorem 1]{North}. Furthermore if $k$ is infinite then $e_1^\A(A) = 0$ if and only if $\A$ is a parameter ideal, see \cite[Theorem 2]{North}.
Narita, see \cite{Nar},  showed that $e_2^\A(A) \geq 0$ (also see \cite[Theorem 3.1]{RV}).  Furthermore if $k$ is infinite and $d = 2$ then $e_2^\A(A) = 0$ if and only if reduction number of $\A^n$ is one for all $n \gg 0$, see \cite{Nar} (also see \cite[Proposition 3.2]{RV}). In particular $G_{\A^n}(A)$ is \CM \ for all $n \gg 0$. Narita, see \cite{Nar}, also gave an example of a ring with $e_3^\A(A) < 0$ (also see \cite[Example 2, p.\ 8]{Mar}). For a nice survey of this area see the article \cite{V}.

The integral closure filtration of  $\A$ is $\{ (\A^n)^* \}_{n \geq 0}$.  The integral closure filtration of $\m$-primary ideals  has inspired plenty of research, see for instance \cite[Appendix 5]{ZS}, \cite{Rees-2}, \cite{Hun}, \cite{It}, \cite{ItN}, \cite{Masuti-1}, \cite{Masuti-2}
\cite{G}, \cite{GHM}, \cite{CPR} and \cite{Ver}.

Itoh proved that ${e_3^\A}^*(A) \geq 0$, see \cite[Theorem 3]{ItN}. He conjectured that if $A$ is Gorenstein and
${e_3^\A}^*(A) = 0$ then $\GA^* = \bigoplus_{n \geq 0} (\A^n)^*/(\A^{n+1})^*$ is \CM. He proved the result when
${\A}^* = \m$, see \cite[Theorem 3]{ItN} (also see \cite[3.3]{CPR}). See \cite{KM} for a geometric view-point of Itoh's conjecture.

An ideal $\A$ is said to be normal if $\A^n$ is integrally closed for all $n \geq 1$.
In a previous paper \cite{Pu-nor} we proved that Itohs conjecture holds when $\A$ is normal and $e_3^\A(A) = 0$. In fact we  proved it more generally  for \CM \ analytically unramified rings. In this paper we prove Itoh's conjecture in general. We prove
\begin{theorem}
\label{main} Let $(A,\m)$ be a \CM \ analytically unramified local ring of dimension $d \geq 3$ and let $\A$ be an $\m$-primary ideal with $e_3^{\A^*}(A) = 0$. Then  $\GA^*$ is \CM.
\end{theorem}
\emph{Note we are not assuming $A$ is Gorenstein in Theorem \ref{main}.}

Our paper on Itoh's conjecture for normal ideals gives us a road-map of a possible proof of the general case. However the details are considerably harder to prove in the general case.

We first describe the technique we used in the normal case and then we discuss the changes that we made in the general case.

\emph{First technique: Dual filtrations:}\\
We need some preliminaries to get to our discussion.
Let $M$ be a finitely generated $A$-module.
Recall an $\A$-\textit{filtration} $\eG = {\{M_n\}}_{n \in \Z}$ on $M$ is a
collection of submodules of $M$ with the properties
\begin{enumerate}
\item
$M_n \supseteq M_{n+1}$ for all $n \in \Z$,
\item
$M_n = M$ for all $n \ll 0$,
\item
$\A M_n \subseteq M_{n+1}$ for all $n \in \Z$.
\end{enumerate}
If
$\A M_n = M_{n+1}$ for all $n \gg 0$ then we say $\eG$ is $\A$-\emph{stable}.

Let $\ra = \bigoplus_{n\in \Z} \A^n$ be the extended Rees algebra of $A$ \wrt \ $\A$.
If $\eG = \{M_n\}_{n \in \Z}$ is an $\A$-stable filtration on $M$ then  set
$\eR(\eG,M) = \bigoplus_{n \in \Z}M_n$ the \emph{extended Rees-module} of $M$
\emph{\wrt }\ $\eG$.
 Notice that $\eR(\eG,M)$ is a finitely generated graded $\ra$-module. If $\eG$ is $\A$-adic then set $\eR(\eG,M) = \eR(\A, M)$.

Set $M^* = \Hom_A(M, A)$.  For $n \in \Z$ set
$$  M_n^* = \{ f \in M^* \mid    f(M) \subseteq \A^n \} \cong \Hom_A(M, \A^n).      $$
It is well-known that $M_{n+1}^* = \A M_n^*$ for all $n \gg 0$.
Note $M_n^* = M^*$ for $n \leq 0$. Set $\eF = \{M_n^*\}_{n \in \Z}$. We called this filtration \emph{ the dual filtration } of $M^*$ \wrt \  $\A$. This filtration is
classical cf.  \cite[p.\ 12]{Ser}.
 \emph{In \cite[4.4]{Pu-nor}    we constructed  an explicit isomorphism}:
$$\Psi_M \colon \eR(\eF,M^{*}) \quad \xar \Hs_{\ra}(\eR(\A, M),\ra ). $$
It was an important observation that the dual filtrations behaves well  \wrt \ the Veronese functor, see \cite[4.8]{Pu-nor}.

For the general case we first prove:
\begin{theorem}\label{dual}
Let $(A,\m)$ be a local Noetherian ring and let $\A$ be an ideal in $A$.
Let $M$ be  a finitely generated $A$-module and let $\eG$ be an $\A$-stable filtration on
$M$. Then there exists an $\A$-stable filtration  $\eH$ on $M^* = \Hom_A(M, A)$ such that
$$  \eR(\eH,M^{*}) \cong  \Hs_\R(\eR(\eG, M),\ra ). $$
\end{theorem}
The isomorphism above is existential. However it is an important ingredient in our proof. We call $\eH$ to be the \emph{dual filtration} on $M^*$ with respect to $\eG$.

In the adic case the good behaviour of dual filtrations with respect to the Veronese functor was just an observation. However in the general case we don't know whether it is true. We prove a special case which is sufficient for our purposes.
\begin{theorem}
\label{dual-Ver} Let $(A,\m)$ be a local Noetherian ring  with $\depth A \geq 2$ and let $\A$ be an $\m$-primary ideal in $A$.
Let $M$ be  a finitely generated $A$-module and let $\eG$ be an $\A$-stable filtration on
$M$ with $\eG_n = M$ for $n \leq 0$ and $\eG_1 \neq M$. Further assume there exists $x\in \A \setminus \A^2$
such that its initial form $x^*$ is $\GA$-regular. Then  for every $l \geq 1$ we have an isomorphism of $\mathcal{R}_{\A^l}(A)$-modules
$$ \Hs_{\ra}(\eR(\eG, M),\ra )^{<l>} \cong \Hs_{\mathcal{R}_{\A^l}(A)}(\eR(\eG^{<l>}, M),\mathcal{R}_{\A^l}(A) ). $$
\end{theorem}
Both Theorems \ref{dual} and \ref{dual-Ver} are considerably more difficult than \cite[4.4, 4.8]{Pu-nor} which we proved in the adic-case. Furthermore the proof of these results is a bit non-intuitive.

\emph{Extension of some results from the adic-case.}\\
(a) By a result of Elias, \cite[2,2]{E}, it is known that $\depth G_{I^n}(A)$ is constant for $n \gg 0$ for any $\m$-primary ideal $I$. We have to extend this result to filtration's. Elias's proof does not extend to filtration's.  We prove
\begin{theorem}
  \label{asymp}
  Let $(A,\m)$ be a \CM \  local ring.  Let $\A$  be an $\m$-primary ideal in $A$. Let $\eF = \{ \A_n \}_{n\geq 0}$ be a multiplicative $\A$-stable filtration.
  Then \\ $\depth G_{\eF^{<m>}}(A) $ is constant for
  $m \gg 0$.
\end{theorem}

(b)  Let $\eF = \{ \A_n \}_{n\geq 0}$ be an asymptotic integral closure filtration with respect to $\A$.  When $\depth A \geq 2$ and $\A_n = I^n$ (for all $n$) for some $\m$-primary ideal in $A$  then it is known that $\depth G_{I^n}(A) \geq 2$ for all $n \gg 0$, see \cite[3.1]{HH}, We need to generalize this results to filtration's which are not adic. We prove
  \begin{theorem}
  \label{hh}
  Let $(A,\m)$ be a \CM \ analytically unramified local ring with $\dim  A \geq 2$.  Let $\A$  be an $\m$-primary ideal in $A$. Let $\eF = \{ \A_n \}_{n\geq 0}$ be asymptotic integral closure filtration \wrt \ $\A$. Then $\depth G_{\eF^{<m>}}(A) \geq 2$ for all $m \gg 0$.
  \end{theorem}

(c) Let $\A$ be an $\m$-primary ideal an let $\eF = \{ \A_n \}_{n \geq 0}$ be a multiplicative $\A$-stable filtration on $A$. Let $d = \dim A$. Set
\[
a_\eF(A) = \max \{ n \mid H^d_{\GA_+}(G_\eF(A))_n \neq 0 \}.
\]
The integer $a_\eF(A)$ is called the $a$-invariant of $G_\eF(A)$
  By a result due to Hoa, \cite{Hoa},   the a-invariant of $G_\A(A)$ is negative if and only if reduction number of $\A^n$ is less than $d = \dim A$ for all $n \gg 0$. We are unable to extend this to filtrations. However the following result is sufficient for us.
\begin{theorem}
\label{Hoa} Let $(A,\m)$ be a Noetherian local ring of dimension $d \geq 1$ and let $\A$ be an $\m$-primary ideal. Let $\eF = \{ \A_n \}_{n\geq 0}$ be a multiplicative $\A$-stable filtration.
Suppose $\eF^{<c>}$ is $J = \A_c$-adic and  that reduction number of $J^n$ is $< d$ for all $n \gg 0$. Then the $a$-invariant of $G_\eF(A)$ is negative.
\end{theorem}

\emph{An extension to filtrations of a technique due to us.} \\
In \cite{Pu5} we introduced a new technique in the adic case to investigate problems relating to associated graded modules.
The extension to filtrations is routine and practically all the proofs goes through.
\s Let $\eF = \{ M_n \}_{n \in \Z}$ be an $\A$-stable filtration on $M$  with $M_n = M$ for $n \leq 0$.  Set $L^\eF(M) = \bigoplus_{n \geq 0} M/M_{n+1}$.
Let $\ta$ be the Rees algebra of $\A$ and let $\Sc(\eF, M) = \bigoplus_{n \geq 0}M_n$ be the Rees module of $M$ with respect to $\eF$.
We note that \\ $M[t]/\Sc(\eF, M) = L^\eF(M)(-1)$.
Thus $L^\eF(M)$ is a $\ta$-module. Note $L^\eF(M)$ is NOT a finitely generated $A$-module.

We need the following result:
\begin{theorem}\label{d2-h1L}
Let $(A,\m)$ be a Noetherian local ring and let $\A$ be an $\m$-primary ideal. Let $M$ be a \CM \ $A$-module of dimension $r \geq 2$ and let \\  $\eF = \{ M_n \}_{n \in \Z}$ be an $\A$-stable filtration on $M$  with $M_n = M$ for $n \leq 0$.
Let $\M$ be the maximal homogeneous ideal of $\ta$ and let $H^i_\M(-)$ denote the $i^{th}$ local cohomology functor with respect to $\M$.
If $\depth G_{\eF^{<m>}}(M) \geq 2$ for some $m$ then $H^1_\M(L^\eF(M))_n = 0$ for $n < 0$.
\end{theorem}

Armed with these techniques the strategy to prove Theorem \ref{main} is mostly similar to that employed to prove Itoh's conjecture for normal ideals.
Here is an overview of the contents of this paper. In section two we introduce notation and discuss a few
preliminary facts that we need. In section three we prove Theorem \ref{dual} and in the next section we prove Theorem \ref{dual-Ver}. In section five we discuss a few preliminary properties of $L^\eF(M)$ that we need and prove Theorem \ref{d2-h1L}. In section six we prove Theorems \ref{asymp} and \ref{hh}. In the next section we prove Theorem \ref{Hoa}. Finally in section eight we prove Itoh's conjecture.
\section{Preliminaries}
In this section we discuss a few preliminaries that we need. In this paper all rings are commutative Noetherian and all modules  (unless stated otherwise) are assumed to
be finitely generated. Let $A$ be a local ring, $\A$ an ideal in $A$ and let $N$ be an $A$-module. Then set
 $\ell(N)$ to be length of $N$ and $\mu(N)$ the number of minimal generators of $N$.

\s \label{mod-N}
Let $\eG = {\{M_n\}}_{n \in \Z}$ be an $\A$-stable filtration on $M$ and let $N$ be a submodule of $M$. By the \emph{quotient filtration} on $M/N$ we mean the filtration   $\ov{\eG} = \{ (M_n + N)/N \}_{\nZ}$.
If $\eG$ is an $\A$-stable filtration on
$M$ then $\ov{\eG}$ is an $\A$-stable filtration on
$M/N$.\textit{ Usually} for us $N =\xb M$ for some $\xb = x_1,\ldots,x_s \in \A$.

\s If $\eG = \{M_n\}_{n \in \Z}$ is an $\A$-stable filtration on $M$ then  set
$\eR(\eG,M) = \bigoplus_{n \in \Z}M_n$ the \emph{extended Rees-module} of $M$
\emph{\wrt }\ $\eG$.
 Notice that $\eR(\eG,M)$ is a finitely generated graded $\ra$-module. If $\eG$ is the usual $\A$-adic filtration then
set \\  $ \eR(\A, M) = \eR(\eG,M)$.

Set $G_\eG(M)=\bigoplus_{n \in \Z} M_n/M_{n+1}$, the \emph{associated graded
module }of $M$ \emph{\wrt} \ $\eG$. Notice $G_\eG(M)$ is a finitely generated graded module over
 $G_{\A}(A)$.
 Furthermore $\eR(\eG,M) /t^{-1}\eR(\eG,M) = G_\eG(M)$.

\s \label{shift-filt}
Let $\eG = {\{M_n\}}_{n \in \Z}$ be an $\A$-stable filtration on $M$ and let $s \in \Z$. By the $s$-\emph{th shift} of $\eG$, denoted by $\eG(s)$ we mean the filtration ${\{ \eG(s)_n\}}_{n \in \Z}$ where
$\eG(s)_n = \eG_{n+s}$. Clearly
$$G_{\eG(s)}(M) = G_\eG (M)(s) \quad \text{and} \quad \R(\eG(s), M) = \R(\eG, M)(s). $$

\s All   filtrations in this paper $\eG = \{M_n\}_{n \in \Z}$ will be \textit{separated} i.e.,
 $\bigcap_{\nZ}M_n = \{0\}$.
This is automatic if  $\A \neq A$ and $\eG$ is $\A$-stable.
If $m$ is a non-zero element of $M$ and if $j$ is the largest integer such that
$ m \in M_j$,
then we let $m^{*}_{\eG}$ denote the image of $m $ in $M_j  \setminus M_{j+ 1}$ and we call it the \textit{initial form} of $m$ \wrt \ $\eG$. If $\eG$ is clear from the context then we drop the subscript $\eG$.

\s \emph{Veronese functor}. Let $\eF = \{M_n \}_{n\in \Z}$ be an $\A$-stable filtration. For $l \geq 1$ set $\eF^{<l>}= \{M_{nl} \}_{n \in \Z}$. Notice $\eF^{<l>}$ is an $\A^l$-stable filtration. We also have an $\ral$-isomorphism $\eR(\eF,M)^{<l>} \cong \eR(\eF^{<l>}, M)$.

\s \label{sup} Let $\eF$ be an $\A$-stable filtration on $M$. We say $x \in \A$ is $M$-superficial with respect to $\eF$ if there exists $c$ such that  $(M_{n+1}  \colon x) \cap M_c = M_n$ for all $n \gg 0$.
If $\grade(\A, M) > 0$ then it can be shown that an $M$-superficial element $x$ is $M$-regular. Furthermore in this case $(M_{n+1} \colon x) = M_n$ for all $n \gg 0$.
The proof of these assertion follows exactly as in the adic case with $M = A$; see \cite[p.\ 6-8]{S}.

\s\label{AtoA'} \textbf{Flat base-change:} In our paper we do many flat changes of rings.
 The general set up we consider is
as follows:

 Let $\phi \colon (A,\m) \rt (A',\m')$ be a flat local ring homomorphism
 with $\m A' = \m'$. Set $\A' = \A A'$ and if
 $N$ is an $A$-module set $N' = N\otimes_A A'$. Set $k = A/\m$ and $k' = A'/\m'$.
If $\eF = \{ M_n \}_{n \in \Z}$ is an $\A$-stable filtration on $M$ then set ${\eF}' = \{ M_n'\}_{n \in \Z}$. It is elementary to show that
${\eF}'$ is an $\A'$-stable filtration on $M'$.

 \textbf{Properties preserved during our flat base-changes:}

\begin{enumerate}[\rm (1)]
\item
$\ell_A(N) = \ell_{A'}(N')$. So
 $\ell_A(M/ M_n)  = \ell_{A'}(M'/M_n')$  for all $n \geq 0$.
\item
$\dim M = \dim M'$ and  $\grade(K,M) = \grade(KA',M')$ for any ideal $K$ of $A$.
\item
$\depth G_{\eF}(M) = \depth G_{{\eF}'}(M')$.
\end{enumerate}

\textbf{Specific flat base-changes:}

\begin{enumerate}[\rm (a)]
\item
$A' = A[X]_S$ where $S =  A[X]\setminus \m A[X]$.
The maximal ideal of $A'$ is $\n = \m A'$.
The residue
field of $A'$ is $k' = k(X)$. Notice that $k'$ is infinite.
\item
$A' =  \widehat{A}$  the completion of $A$ \wrt \ the maximal ideal.
\item
Applying (a) and (b) successively ensures that $A'$ is complete with $k'$ infinite.
\item
$A' = A[X_1,\ldots,X_n]_S$ where $S =  A[X_1,\ldots,X_n]\setminus \m A[X_1,\ldots,X_n]$.
The maximal ideal of $A'$ is $\n = \m A'$.
The residue
field of $A'$ is $l = k(X_1,\ldots,X_n)$. Notice that if $I$ is integrally closed then $I'$ is
also integrally closed.
When $\dim A \geq 2$, $A$ is  \CM \ and $I$ is $\m$-primary, say $I = (a_1,\cdots,a_n)$ then in \cite[Corollary 2]{Ciu} it is proved that
the element $y = \sum_{i = 1}^{n} a_iX_i$ in $A'$ is $I'$-superficial and
the $A'/(y)$ ideal $J = I'/(y)$ has the property that
$$J_n^* = \ov{((I^n)')^*} \quad \text{for all} \ n \gg 0.$$
\end{enumerate}

\s \label{lying-above} We will need the following result. Let $\phi \colon (B,\n) \rt (A, \m)$ be a local map such that $A$ is finitely generated $B$-module (via $\phi$). Consider
$R = B[X_1,\cdots,X_n]$ and $S = A[X_1,\cdots, X_n]$ and let $\psi \colon R \rt S$ be the map induced by $\phi$. Note $S$ is a finite $R$-module via $\psi$.
Let $P = \n R$ and $Q = \m S$.  Set $S_P = W^{-1}S$ where $W = R \setminus P$. Then $S_P = S_Q$. To see this note that $Q$ is the only prime in $S$ lying over $P$. The result follows from an exercise problem
in \cite[Exercise 9.1]{Ma}.

\s \label{redNO} We assume $k = A/\m$ is infinite. Let $\C = (x_1,\ldots,x_l)$ be  a minimal reduction of
$\A$ \wrt \ $M$.  We denote by  $\red_\C(\A, M) := \min \{ n  \mid \  \C \A^n M= \A^{n+1}M \} $
the \textit{reduction number} of $\A$ \wrt \ $\C$ and $M$. Let
 $$\red(\A, M) = \min \{ \red_\C(\A, M) \mid \C
\text{\ is  a minimal reduction of \ }  \A \}$$
 be the \textit{reduction number } of $M$ \wrt \ $\A$. Set
$\red_\A(A) = \red(\A, A)$.

\s \label{gen-cm} Let $R= \bigoplus_{n\geq 0}R_n$ be a standard graded algebra over an Artin local ring $(R_0,\m_0)$. Set $\M = \m_0 \oplus (\bigoplus_{n\geq 1}R_n)$. We say that $R$ is generalized \CM \ if $H^i_{\M}(R)$ has finite length for $i < \dim R$.

\s \emph{CI-approximation:} Let $(A,\m)$ be a  local ring and let $\A$ be an $\m$-primary ideal. By a CI-approximation of $A$ \wrt \ $\A$ we mean a  local ring $(B, \n)$ and an ideal $\B$ of $B$ and a local ring homomorphism $\phi \colon B \rt A$ such that
\begin{enumerate}
  \item $A$ is a finite $B$-module (via $\phi$).
  \item $\dim B = \dim A$.
  \item $\B A = \A$.
  \item $B$ and $G_\B(B)$ are complete intersections.
\end{enumerate}
We say $(B,\n, \B, \phi)$ is a CI-approximation of $(A,\m,\A)$.
In \cite[1.5, 5.4]{Pu-nor} we proved the following result
\begin{theorem}
  \label{ci-approx} Let $(A,\m)$ be a complete local  ring and let $\A = (a_1, \ldots, a_s)$ be an $\m$-primary ideal.  Then there exists a CI-approximation $(B,\n, \B, \phi)$ of $(A,\m, \A)$ with $\B = (b_1, \ldots, b_s)$ and $\phi(b_i) = a_i$. Furthermore if
  $\mu(\A) > \dim A$ and $r \geq 1$ an integer then  there exists   a CI-approximation $(B_r,\n, \B_r, \phi)$ of $(A,\m, \A)$ with $\red_{\B_r}(B_r) >  r$.
\end{theorem}

\section{Proof of Theorem \ref{dual}}
In this section we give a proof of Theorem \ref{dual}. We first prove:

\begin{theorem}\label{repr}
Let $(A,\m)$ be a local  ring and let $\A$ be an ideal in $A$.
Let $\Ee$ be  a finitely generated $\ra$-module such that $t^{-1}$ is $\Ee$-regular.
Then there exists a finitely generated $A$-module $M$ and an $\A$-stable filtration $\eF$ on $M$ such that $\eR(\eF, M) \cong \Ee$.
\end{theorem}
\begin{proof}
Let $\Ee = \bigoplus_{n \in \Z}\Ee_n$ where $\Ee_n$ are finitely generated $A$-modules. As $t^{-1}$ is $\Ee$-regular we have an injective map $\Ee(+1) \xrightarrow{t^{-1}} \Ee$. It follows that we may consider $\Ee_{n}$ to be an $A$-submodule of $\Ee_{n-1}$ for all $n \in \Z$. We note that $X = \Ee/t^{-1}\Ee$ is a finitely generated $G_\A(A)$-module and so $X_n = 0$ for all $n \ll 0$.
After shifting we may assume $X_n = 0$ for all $n \leq -1$ and $X_0 \neq 0$. Thus it follows that $\Ee_n = \Ee_0$ for $n \leq -1$ and $\Ee_1 \neq \Ee_0$.

Set $M = \Ee_0$. Set $M_n = \image(\Ee_n \xrightarrow{t^{-n}} \Ee_0)$ for $n \geq 1$ and $M_n = M$ for $n \leq 0$. It is clear that $M_n \subseteq M_{n-1}$ for all $n \in \Z$. We note that $M_n = M$ for $n < 0$ and $M_1 \neq M_0$.
We first show that $\eF = \{ M_n \}_{n \in \Z}$ is an $\A$-filtration on $M$. Let $x \in \A$. Set $u = xt \in \ra_1$. Let $\alpha \in M_n$.
If $n \leq -1$ then note $x\alpha \in M_{n+1}$. Let $n \geq 0$.
Let $\beta \in \Ee_n$ such that $t^{-n}\beta = \alpha$. Then note $t^{-n-1}(u\beta) = x\alpha$. So $x\alpha \in M_{n+1}$. Thus $\eF$ is an $\A$-filtration on $M$.

Define $\Psi \colon \eE \rt  \eR(\eF, M)$ as follows. Set $\Psi_n = $ identity on $\Ee_0 = M$ for $n \leq 0$. For $n \geq 1$ set $\Psi_n(\alpha) = t^{-n}(\alpha)$.
Clearly $\Psi$ is a bijective $A$-linear map.
 We show $\Psi$ is $\ra$-linear.
Clearly $\Psi$ is additive. Also it is clear that $\Psi(t^{-1}\alpha) = t^{-1}\Psi(\alpha)$ for $\alpha \in \Ee_n$.
Let $x \in \A$. Set $u = xt \in \ra_1$. Let $\alpha \in \Ee_n$. If $n \leq -1$ then it is clear that $\Psi_{n+1}(u\alpha) = u\alpha = u\Psi_n(\alpha)$.
For $n \geq 0$ notice
\[
\Psi_{n+1}(u\alpha) = t^{-n -1}(u\alpha) = t^{-n}(x\alpha) = xt^{-n}(\alpha) = u\Psi_n(\alpha).
\]
Thus $\Psi$ is a $\ra$-linear isomorphism. It follows that $\eF$ is an $\A$-stable filtration on $M$.
\end{proof}

We now give
\begin{proof}[Proof of Theorem \ref{dual}]
We may assume  (after possibly shifting) that $\eG_n = M$ for all $n \leq 0$ and $\eG_1 \neq M$. We may also assume that $\A \eG_n = \eG_{n+1}$ for all $n \geq n_0$.

We first show that
\[
\Hom_A(M, A) = 0  \quad \text{if and only if} \quad \Hs_\R(\eR(\eG, M),\ra ) = 0.
\]
First assume that $\Hom_A(M, A) = 0$. Then $\grade \ann_A M > 0$. Let $c \in \ann_A M$ be $A$-regular. It is elementary to show that $ct^{-1}$ is $\ra$-regular and that $ct^{-1} \in \ann_{\ra} \eR(\eG, M)$. Thus $\Hs_\R(\eR(\eG, M),\ra ) = 0$.

Conversely assume  $\Hs_\R(\eR(\eG, M),\ra ) = 0$. Let \\ $u = at^l \in \ann_{\ra} \eR(\eG, M)$ be $\ra$-regular. After possibly multiplying with sufficiently high power of $t^{-1}$ we may assume that $l \leq -1$. Then note that $a$ is necessarily $A$-regular. Furthermore if $m \in M$ then as $u (m t^0) = 0$ it follows that $a m = 0$. So $a \in \ann_A M$. Thus $\Hom_A(M, A) = 0$.

Thus result holds if $\Hom_A(M, A) = 0$. Next assume that $\Hom_A(M, A) \neq 0$. Equivalently $\Hs_\R(\eR(\eG, M),\ra ) \neq 0$. Note $t^{-1}$ is
$\Hs_\R(\eR(\eG, M),\ra )$-regular. So by \ref{repr} there exists a finitely generated $A$-module $L$ and a $\A$-stable filtration on $L$ such that $\eR(\eF, L) \cong \Hs_\R(\eR(\eG, M),\ra )$.

 Assume $\eF_n = L$ for $n \leq m$. Fix $c \geq \max\{ -m, n_0 +1 \}$. Let \\
 $\psi \in \Hs_\R(\eR(\eG, M),\ra )_{-c}$. Then note $\psi(Mt^0) \subseteq At^{-c}$. So
 we have a well defined $A$-linear map $\psi^\sharp \colon M \rt A$.
 Consider the map
 \begin{align*}
 \theta \colon \Hs_\R(\eR(\eG, M),\ra )_{-c} &\rt \Hom_A(M, A), \\
 \psi &\mapsto \psi^\sharp.
 \end{align*}
 Clearly $\theta$ is $A$-linear. We show $\theta$ is an isomorphism.

 We first show $\theta$ is injective. Suppose $\theta(\psi) = \psi^\sharp = 0$. Then $\psi(Mt^0) = 0$. Let $at^{-n} \in \eG_{-n}$ for $n > 0$. Then
 \[
 \psi(at^{-n}) = \psi(at^0)t^{-n} = 0.
 \]
Let $bt^n \in \eG_{n}$ for $n > 0$. Then note
\[
t^{-n}\psi(bt^{n})  = \psi(bt^0) = 0.
\]
It follows that $\psi(bt^n) = 0$. Thus $\psi = 0$. Therefore $\theta$ is injective.

 Next we show $\theta$ is surjective. Let $f \in \Hom_A(M, A)$. \\
 Claim $f(M_m) \subseteq \A^{m-c}$ for all $m \in \Z$.\\
This is trivial if $m \leq c$. We prove the result for $m \geq c$ by induction on $m$.
For $m = c$ this is obvious. Assume the result for $k \geq c$ and prove for $k+1$.
Note $M_{k+1} = \A M_k$. So we have
\[
f(M_{k+1}) \subseteq \A f(M_k)\subseteq \A \A^{k-c} = \A^{k+1 -c}.
\]
Thus by induction Claim-1 follows.

Define $\psi_f \colon \eR(\eG, M) \rt \ra$ by
$\psi_f(a t^m) = f(a)t^{m-c}$. By Claim-1, $\psi_f$ is well-defined. Also clearly $\psi_f$ is $A$-linear and $\psi_f^\sharp = f$. It remains to prove $\psi_f$ is $\ra$-linear.
Let $u = at^m \in \eG_m$. Then we have
\[
\psi_f(t^{-1}u) = f(a)t^{m-1-c} = t^{-1}\psi_f(u).
\]
Let $x \in \A$. Then we have
\[
\psi_f(xt u) = f(xa)t^{m+1-c} = xf(a)t^{m+1-c}  = xtf(a)t^{m-c} = xt\psi_f(u).
\]
Thus $\psi_f$ is $\ra$-linear. As $\theta(\psi_f) = f$ it follows that $\theta$ is surjective.
So we have
\[
L \ \cong \  \Hs_\R(\eR(\eG, M),\ra )_{-c} \ \stackrel{\theta}{\cong} \ \Hom_A(M, A).
\]
The result follows.
\end{proof}
\section{Proof of Theorem \ref{dual-Ver}}
In this section we give proof of Theorem \ref{dual-Ver}.
We first make a few observations.
\s \label{maps}Let $M$ be an  $A$-module. Let $\eG$ be an $\A$-stable filtration on $M$. Assume $\eG_n = M$ for all $n \leq 0$ and $\eG_{1} \neq M$. It is easily verified that $\phi \in \Hs_\R(\eR(\eG, M),\ra )_{c}$ determines $A$-linear maps $\psi_i \colon M_i \rt A$ for all $i \in \Z$
defined as $\psi(mt^i) = \psi_i(m)t^{c+i}$ (for $m \in M_i$). Furthermore as $\psi$ is $\ra$-linear the maps $\psi_i $ for $i \in \Z$ satisfy the following properties
\begin{enumerate}
\item $\psi_i(M_i) \subseteq \A^{c+ i}$.
\item Fix $i \in \Z$ and $m \in M_i$
\begin{enumerate}
\item
For all $j \geq 0$  and all $i \in \Z$ we have 
\[
\psi_i(m) = \psi_{i-j}( m).
\]
\item
Fix $l \geq 1$ and $x \in \A^l$. Then we have
\[
x\psi_i(m) = \psi_{i+l}(xm).
\]
\end{enumerate}
\end{enumerate}
\textit{Conversely} it is evident that if we have a collection $(\psi_i \colon M_i \rt A)_{i \in \Z}$ of $A$-linear maps which satisfy (1) and (2) then we have an $\ra$-linear map $\psi \in \Hs_\R(\eR(\eG, M),\ra )_{c}$ defined as $\psi(mt^i) = \psi_i(m)t^{c+i}$.

\begin{lemma}\label{d2}
Let $(A,\m)$ be a local ring and let $M$ be an $A$-module. Assume $\depth  A \geq 2$. Let $N$ be an $A$-submodule of $M$ such that $M/N$ has finite length as an $A$-module. Then the natural map $\Hom_A(M, A) \rt \Hom_A(N, A)$ is an isomorphism.
\end{lemma}
\begin{proof}
We have an exact sequence $0 \rt N \rt M \rt M/N \rt 0$. The result follows by observing
$\Ext^i_A(M/N, A) = 0$ for $i = 0, 1$.
\end{proof}
Next we show:
\begin{theorem}\label{pre-ver}
Let $(A,\m)$ be a local ring  with $\depth A \geq 2$ and let $\A$ be an $\m$-primary ideal in $A$.
Let $M$ be  an $A$-module and let $\eG$ be an $\A$-stable filtration on
$M$ with $\eG_n = M$ for $n \leq 0$ and $\eG_1 \neq M$. Then  for all $c  \in \Z$
$$ \Hs_{\ra}(\eR(\eG, M),\ra )_{c} \cong \{ \psi \in \Hom_A(M, A) \mid \psi(M_i) \subseteq \A^{c + i}  \ \text{for all $i \in \Z$}\}. $$
\end{theorem}
\begin{proof}
  Set $W  = \{ \psi \in \Hom_A(M, A) \mid \psi(M_i) \subseteq \A^{c + i} \}$. If $\psi \in W$ then let $\psi_i$ be the restriction of $\psi$ on $M_i$. Then $(\psi_i)_{i \in \Z}$ satisfies properties (1), (2) in \ref{maps}. So we have a map $\psi^\sharp \in \Hs_{\ra}(\eR(\eG, M),\ra )_{c} $.

  Conversely if $\psi^\sharp  \in \Hs_{\ra}(\eR(\eG, M),\ra )_{c} $ then consider the maps \\ $\psi_i \colon M_i \rt A$ defined by for $m \in M_i$ we have $\psi^\sharp(mt^i) = \psi_i(m)t^{c + i}$. Then $(\psi_i)_{i \in \Z}$ satisfies properties (1) and (2)  in \ref{maps}. By \ref{d2} we may consider $\psi_i$ to be a restriction of $\phi_i \in \Hom_A(M, A)$ for all $i \in \Z$. We show $\phi_i = \phi_0$.
  For $j > 0$ by (1) we have $\psi_{-j}(m) = \psi_0(m)$ for $m \in M = M_0 = M_{-j}$. So $\phi_0 = \phi_{-j}$ for all $j > 0$.
  Now let $l > 0$. Let $x \in \A^l$ be $A$-regular. Then
  by (2) we have
  $x\psi_0(m) = \psi_l(xm)$. But $\psi_l(xm) = \phi_l(xm) = x\phi_l(m)$. It follows that $\phi_l = \phi_0$.
  The result follows.
\end{proof}
We now give:
\begin{proof}[Proof of Theorem \ref{dual-Ver}]
By \ref{pre-ver} we have
\begin{align*}
  U &= \Hs_{\ra}(\eR(\eG, M),\ra )^{<l>}_n  =  \Hs_{\ra}(\eR(\eG, M),\ra )_{nl} \\
   &=  \{ \psi \in \Hom_A(M, A) \mid \psi(M_i) \subseteq \A^{nl + i} \text{for all $i \in \Z$} \} ,
\end{align*}
and
\begin{align*}
  V &= \Hs_{\mathcal{R}_{\A^l}(A)}(\eR(\eG^{<l>}, M),\mathcal{R}_{\A^l}(A) )_n\\
   &= \{ \psi \in \Hom_A(M, A) \mid \psi(M_{il}) \subseteq \A^{nl + il} \text{for all $i \in \Z$ } \}.
\end{align*}
Clearly $U \subseteq V$.

 Conversely let $f \in V$. Then $f(M = M_0) \subseteq  \A^{nl}$. Then note for $j < 0$ we have $M_j = M$ and $f(M_j) \subseteq \A^{nl + j}$.
 Now assume $j > 0$. Let $j + s = lr$. Let $u \in M_j$. 
 Let $x \in \A \setminus \A^2$ be such that its initial form $x^*$ is $G_\A(A)$-regular.   Also note $(\A^m \colon x) = \A^{m-1}$ for all $m \geq 1$.
 As $x^su \in M_{lr}$, so  $f(x^su) \in \A^{nl + lr}$. Thus $x^sf(u) \in \A^{nl + lr}$. So
 $$f(u) \in (\A^{nl + lr} \colon x^s) = \A^{nl + lr -s} = \A^{nl + j}.$$
 Thus $V \subseteq U$. The result follows.
\end{proof}

\section{The $L$-construction for filtration's}\label{Lprop}
In \cite{Pu5} we introduced a new technique to investigate problems relating to associated graded modules. This was done for the adic-case. The technique can be extended to filtrations and analogous properties can be proved (with the same proofs).
In this section we collect all the relevant results which we proved in \cite{Pu5} in the adic-case. Throughout this section
$(A,\m)$ is a  local ring with infinite residue field, $M$ is a \emph{\CM }\ module of dimension
$r \geq 1$ and $\A$ is an $\m$-primary ideal. Furthermore $\eF = \{ M_n \}_{n \in \Z}$ is an $\A$-stable filtration on $M$ with $M_n = M$ for $n \leq 0$. We \emph{do not} insist that $M_1 \neq M$.

\s \label{mod-struc} Set $\ta = A[\A t]$;  the Rees Algebra of $\A$.
Let $\Sc(\eF, M) = \bigoplus_{n \geq 0}M_n$ be the Rees module of $M$ with respect to $\eF$.
We note that  $M[t]/\Sc(\eF, M) = L^\eF(M)(-1)$.
Thus $L^\eF(M)$ is a $\ta$-module. Note $L^\eF(M)$ is NOT a finitely generated $\ta$-module.

 \s Set $\M = \m\oplus \ta_+$. Let $H^{i}(-) = H^{i}_{\M}(-)$ denote the $i^{th}$-local cohomology functor \wrt \ $\M$. Recall a graded $\ta$-module $E$ is said to be
*-Artinian if
every descending chain of graded submodules of $E$ terminates. For example if $U$ is a finitely generated $\Sc$-module then $H^{i}(U)$ is *-Artinian for all
$i \geq 0$.

\s \label{zero-lc} In \cite[4.7]{Pu5} we proved gave an explicit description of $H^0(L^\eF(M))$ in the adic case. In general we have $H^0(L^\eF(M))$  is an $\ta$-module of finite length. To see this let $x \in \A$ be $M$-superficial with respect to $\eF$.
Then by \ref{sup}   we have $(M_{n+1} \colon x) = M_n$ for all $n \gg 0$ say from $n \geq n_0$. Note $xt \in \M_1$ and so we have $H^0(L^\eF(M))_n = 0$ for all $n \geq n_0 + 1$.

\s \label{Artin}
For   $0 \leq i \leq  r - 1$
\begin{enumerate}[\rm (a)]
\item
$H^{i}(L^\eF(M))$ are  *-Artinian $\ta$-modules; see \cite[4.4]{Pu5}.
\item
$H^{i}(L^\eF(M))_n = 0$ for all $n \gg 0$; see \cite[1.10]{Pu5}.
\item
 $H^{i}(L^\eF(M))_n$  has finite length
for all $n \in \mathbb{Z}$; see \cite[6.4]{Pu5}.
\item
$\ell(H^{i}(L^\eF(M))_n)$  coincides with a polynomial for all $n \ll 0$; see \cite[6.4]{Pu5}.
\end{enumerate}

\s \label{I-FES} The natural maps $0\rt M_n/M_{n+1} \rt M/ M_{n+1}\rt M/M_n \rt 0 $ induce an exact
sequence of $\Sc$-modules
\begin{equation}
\label{dag}
0 \xar G_{\eF}(M) \xar L^\eF(M) \xrightarrow{\Pi} L^\eF(M)(-1) \xar 0.
\end{equation}
We call (\ref{dag}) \emph{the first fundamental exact sequence}.  We use (\ref{dag}) also to relate the local cohomology of $G_\eF(M)$ and $L^\eF(M)$.
\s \label{II-FES} Let $x$ be  $M$-superficial \wrt \ $\eF$ and set  $N = M/xM$ and $u =xt \in \ta_1$. Let $\ov{\eF} = \{ (M_n + xM)/xM \}_{n \in \Z}$ be the quotient filtration on $N$. Notice $$L^\eF(M)/u L^\eF(M) = L^{\ov{\eF}}(N)$$
For each $n \geq 1$ we have the following exact sequence of $A$-modules:
\begin{align*}
0 \xar \frac{M_{n+1}\colon x}{M_n} \xar \frac{M}{M_n} &\xrightarrow{\psi_n} \frac{M}{M_{n+1}} \xar \frac{N}{N_{n+1}} \xar 0, \\
\text{where} \quad \psi_n(m + M_n) &= xm + M_{n+1}.
\end{align*}
This sequence induces the following  exact sequence of $\ta$-modules:
\begin{equation}
\label{dagg}
0 \xar \Bcal^{\eF}(x,M) \xar L^{\eF}(M)(-1)\xrightarrow{\Psi_u} L^{\eF}(M) \xrightarrow{\rho^x}  L^{\ov{\eF}}(N)\xar 0,
\end{equation}
where $\Psi_u$ is left multiplication by $u$ and
\[
\Bcal^{\eF}(x,M) = \bigoplus_{n \geq 0}\frac{(M_{n+1}\colon_M x)}{M_n}.
\]
We call (\ref{dagg}) the \emph{second fundamental exact sequence. }

\s \label{long-mod} Notice  $\ell\left(\Bcal^{\eF}(x,M) \right) < \infty$. A standard trick yields the following long exact sequence connecting
the local cohomology of $L^\eF(M)$ and
$L^{\ov{\eF}}(N)$:
\begin{equation}
\label{longH}
\begin{split}
0 \xar \Bcal^{\eF}(x,M) &\xar H^{0}(L^{\eF}(M))(-1) \xar H^{0}(L^{\eF}(M)) \xar H^{0}(L^{\ov{\eF}}(N)) \\
                  &\xar H^{1}(L^{\eF}(M))(-1) \xar H^{1}(L^{\eF}(M)) \xar H^{1}(L^{\ov{\eF}}(N)) \\
                 & \cdots  \\
               \end{split}
\end{equation}

\s \label{Artin-vanish} We will use the following well-known result regarding *-Artinian modules quite often:

Let $V$ be a *-Artinian $\ta$-module.
\begin{enumerate}[\rm (a)]
\item
$V_n = 0$ for all $n \gg 0$.
\item
If $\psi \colon V(-1) \rt V$ is a monomorphism then $V = 0$.
\item
If $\phi \colon V \rt V(-1)$ is a monomorphism then $V = 0$.
\end{enumerate}

\begin{remark}\label{int-cls}
If $\eF$ is the integral closure filtration of $\A$ then it is well-known that $H^0(\GA^*) = 0$. By \ref{Artin-vanish} it follows that in this case $H^0(L^\eF(A)) = 0$.
\end{remark}

\s \label{asympCM}
Next we prove  s a slight extension of Theorem \ref{d2-h1L}. This was proved for $I$-adic filtrations in \cite[9.2]{Pu5}.
\begin{theorem}\label{depth2}
Let $(A,\m)$ be a Noetherian local ring with infinite residue field and let $\A$ be an $\m$-primary ideal. Let $M$ be a \CM \ $A$-module with $\dim M = r \geq 2$
Let $\eF = \{ M_n \}_{n \in \Z}$ be an $\A$-stable filtration with $M_n = M$ for $n \leq 0$. If the associated graded module of the Veronese $G(\eF^{< c >}, M)$ has depth $\geq 2$ for some $c \geq 1$ then
\begin{enumerate}[\rm (1)]
  \item $H^1(L^\eF(M))_n = 0$ for $n < 0$.
  \item $H^1(G_\eF(M))_n = 0$ for $n < 0$.
\end{enumerate}
\end{theorem}
\begin{proof}
(1)  For all $l \geq 1$ we have
\[
\left(L^\eF(M))(-1)\right)^{<l>} = \bigoplus_{n \geq 0} M/M_{nl} =  L^{\eF^{<l>}}(M)(-1).
\]
As $\depth G_\eF^{< c >} (M) \geq 2$ it follows from  \ref{Artin}; (\ref{dag}) and \ref{Artin-vanish} that
$H^i(L^ {\eF^{<c>}}(M)) = 0$ for $i = 0,1$. As the Veronese functor commutes with local cohomology we get that
$H^i(L^\eF(M))_{nc-1} = 0$ for all $n \in \Z$. In particular $H^1(L^\eF(M))_{-1} = 0$. Let $x$ be $M$-superficial \wrt \ $\eF$.
Set $N = M/xM$. Then by  (\ref{longH}) we have  an exact sequence for all $n \in \Z$
\[
H^0(L^{\ov{\eF}}(N))_n \rt H^1(L^\eF(M))_{n-1} \rt H^1(L^\eF(M))_{n}.
\]
As $H^0(L^{\ov{\eF}}(N))_n = 0$ for $n < 0$ and $H^1(L(\eF, M))_{-1} = 0$;  an easy induction yields $H^1(L(\eF, M))_{n} = 0$ for $n < 0$.

(2) By  (\ref{dag}) and as $H^0(L^\eF(M))_n = 0$ for $n < 0$ we get that $H^1(G_\eF(M))_n = 0$ for $n < 0$.
\end{proof}

\section{A generalization of a result due to Elias and an application to asymptotic integral closure filtration}

By a result of Elias, \cite[2.2]{E}, it is known that $\depth G_{I^n}(A)$ is constant for $n \gg 0$ for any $\m$-primary ideal $I$. We have to extend this result to filtrations. Elias's proof regreatabbly does not extend to filtrations.  We restate Theorem \ref{asymp} for the convenience of the reader. In fact we prove it  for modules.
\begin{theorem}
  \label{asymp-m}
  Let $(A,\m)$ be a  local ring.  Let $\A$  be an $\m$-primary ideal in $A$.
Let $M$  be a \CM \ $A$-module of dimension $r$.
  Let $\eF = \{ M_n \}_{n\in \Z}$ be an $\A$-stable filtration on $M$ with $M_n = M$ for all $n \leq 0$.
  Then \\ $\depth G_{\eF^{<m>}}(M) $ is constant for
  $m \gg 0$.
\end{theorem}
\begin{proof}
We have nothing to show if $r = 0$. So assume $r \geq 1$.
Throughout we take local cohomology with respect to $\M$ the maximal homogeneous ideal of $\ta$.
Let
$$p = \min \{ i | i < r, \ \text{and either} \ H^i(L^\eF(M))_{-1} \neq 0 \ \text{or} \ \ell(H^i(L^\eF(M))) = \infty \}. $$
Note $p \geq 1$. For $i = 0, \ldots, p -1$ set
\begin{align*}
a_i(M) &= \inf \{ n \mid H^i(L^\eF(M))_{n} \neq 0 \}, \\
 b_i(M) &= \sup \{ n \mid H^i(L^\eF(M))_{n} \neq 0 \}.
\end{align*}
Furthermore set
\[
c_i(M) = \begin{cases}
       1   &\text{if} \  H^i(L^\eF(M)) = 0,\\
       b_i(M) - a_i(M) + 2  & \ \text{otherwise}.
               \end{cases}
 \]
 Let $amp(M) = \max \{ c_i  \mod 0 \leq i \leq p-1 \}$.
 We

 Claim: For all $m \geq amp(M)$ we have $\depth G_{\eF^{<m>}}(M) = p$.\\
 We note that for all $m \geq amp(M)$ we have for $i = 0, \ldots, p-1$
 \[
 H^i(L^{\eF^{<m>}}(M))(-1) \cong \left(  H^i(L^\eF(M))(-1)   \right)^{<m>} = 0.
 \]
 By \ref{dag} it follows that $\depth G_{\eF^{<m>}}(M) \geq p$.
If $p = r$ then we are done. Assume $p < r$.
 Suppose if possible 
 $\depth G_{\eF^{<m>}}(M) > p$ for some $m > amp(M)$. Then as
 \[
 0 = H^p(L^{\eF^{<m>}}(M)(-1)) \cong \left(  H^p(L^\eF(M))(-1)    \right)^{<m>},
 \]
 it follows that  $H^p(L^\eF(M))_{-1} = 0$ and $H^p(L^\eF(M))_{mn - 1} = 0$ for all $n \in \Z$. As $n \mapsto \ell(H^p(L^\eF(M))_{n}$ coincides with a polynomial for $n \ll 0$ (see
 \ref{Artin}(d)), it follows that $H^p(L^\eF(M))_{n} = 0$ for $n \ll 0$. Thus $\ell(H^p(L^\eF(M))) < \infty$. This contradicts the definition of $p$.
 The result follows.
\end{proof}
We will need the following result:
\begin{theorem}
  \label{gcm-N}
  Let $(A,\m)$ be a local ring.  Let $\A$  be an $\m$-primary ideal in $A$.
Let $M$  be a \CM \ $A$-module of dimension $r$.
  Let $\eF = \{ M_n \}_{n\in \Z}$ $\A$-stable filtration on $M$ with $M_n = M$ for all $n \leq 0$.
  If  $ G_{\eF^{<m>}}(M) $ is  \CM \ for some $m \geq 1$ then $G_\eF(M)$ is generalized \CM.
\end{theorem}
\begin{proof} We have nothing to prove if $r = 0$. So assume $r > 0$
By \ref{Artin-vanish} and \ref{dag} it follows that
$H^i(L^{\eF^{<m>}}(M)) = 0$ for $i = 0, \ldots, r-1$. We note that for $i < r$ we have
 \[
0 =  H^i(L^{\eF^{<m>}}(M)(-1)) \cong \left(  H^i(L^\eF(M))(-1)   \right)^{<m>}.
 \]
 It follows that $H^i(L^\eF(M))_{mn - 1} = 0$ for all $n \in \Z$ and for $i = 0, \ldots,r-1$.  Fix $i < r$. As $n \mapsto \ell(H^i(L^\eF(M))_{n}$ coincides with a polynomial for $n \ll 0$ (see
 \ref{Artin}(d)), it follows that $H^i(L^\eF(M))_{n} = 0$ for $n \ll 0$. Thus $\ell(H^i(L^\eF(M))) < \infty$ for all $i < r$. By \ref{dag} it follows that  $\ell(H^i(G_\eF(M))) < \infty$ for all $i < r$.  Thus $G_\eF(M)$ is generalized \CM.
\end{proof}

\s \label{normal}
Recall an ideal $I$ is said to be asymptotically normal if $I^n$ is integrally closed for all $n \gg 0$.
If $I$ is a asymptotically normal  $\m$-primary ideal  and $\dim A \geq 2$ then by a result of Huckaba and Huneke \cite[3.1]{HH}, $\depth G_{I^l}(A) \geq 2$
for all $l \gg 0$(also see \cite[7.3]{Pu5}). It  also follows from \cite[9.2]{Pu5} that in this case $H^1(L)_n = 0$ for $n < 0$. In particular $\ell(H^1(L)) < \infty$ (here $L = L^I(A)$). We need to generalize this result for asymptotic integral closure filtration's.
We give a proof of Theorem \ref{hh}. We restate it for the convenience of the reader.
 \begin{theorem}
  \label{hh-m}
  Let $(A,\m)$ be a \CM \ analytically unramified local ring with $\dim  A \geq 2$.  Let $\A$  be an $\m$-primary ideal in $A$. Let $\eF = \{ \A_n \}_{n\geq 0}$ be an asymptotic integral closure filtration with respect to $\A$.  Then $\depth G_{\eF^{<m>}}(A) \geq 2$ for all $m \gg 0$.
  \end{theorem}
\begin{proof}
We note that some Veronese of $\ta$ is standard graded (say $\ta^{<m>} = A[Jt]$. Then $J$ is a asymptotically normal ideal. It follows that $\depth G_{J^n}(A) \geq 2$ for $n \gg 0$.
It follows that  $\depth G_{\eF^{<mn>}}(A) \geq 2$ for all $n \gg 0$. By Theorem \ref{asymp-m} it follows that  $\depth G_{\eF^{<n>}}(A) \geq 2$ for all $n \gg 0$.
\end{proof}

\section{An extension of a result due to Hoa}
In this section we give an extension of a result of Hoa, \cite{Hoa},  which is sufficient for our needs.
We restate it for the convenience of the reader.
\begin{theorem}
\label{Hoa-m} Let $(A,\m)$ be a Noetherian local ring of dimension $d \geq 1$ and let $\A$ be an $\m$-primary ideal. Let $\eF = \{ \A_n \}_{n\geq 0}$ be a multiplicative $\A$-stable filtration.
Suppose $\eF^{<c>}$ is $J = \A_c$-adic and  that reduction number of $J^n$ is $< d$ for all $n \gg 0$. Then the $a$-invariant of $G_\eF(A)$ is negative.
\end{theorem}
\begin{proof}
By our assumption we have that the $a$-invariant of $G_J(A)$ is negative. Suppose if possible
the $a$-invariant of $G_\eF(A)$ is $\geq 0$. Then as local cohomology behaves very well with respect to the Veronese functor we have $H^d(G_\eF(A)^{<c>})_0 \neq 0$.

Notice we have surjective maps
\[
\A_{cn}/\A_{c(n+1)} \rt \A_{cn}/\A_{cn+1} \rt 0 \quad \text{for all}  \ n \geq 1.
\]
So we have a surjective ring homomorphism $G_J(A) \rt G_\eF(A)^{<c>}$. As dimensions of both the rings considered is $d$ we get a surjective map
\[
H^d(G_J(A)) \rt H^d(G_\eF(A)^{<c>}) \rt 0,
\]
which implies that $H^d(G_J(A))_0 \neq 0$, a contradiction. The result follows.
\end{proof}

\section{ Proof of Itoh's conjecture}
In this section we give a proof of Itoh's conjecture.  We restate it for the convenience of the reader.
\begin{theorem}\label{itoh}
Let $(A,\m)$ be a \CM \ analytically unramified local ring of dimension $d \geq 3$ and let $\A$ be an $\m$-primary ideal with $e_3^{\A^*}(A) = 0$. Then  $\GA^*$ is \CM.
\end{theorem}
The proof is mostly similar to the proof in \cite[section 7]{Pu-nor}. However it is different in quite a few places. So we are forced to give the whole proof.

\s
Recall a multiplicative $\A$-stable filtration $\eF = \{ \A_n \}_{n \geq 0}$ is said to be asymptotically  integral closure filtration with respect to $\A$   if $\A_n = \A_n^*$ for all $n \gg 0$.
The crucial result to prove Itoh's conjecture is the following:
\begin{lemma}\label{crucial-itoh}
Let $(A,\m)$ be an analytically unramified  \CM \    local ring of dimension $3$ and let
 $\A$ be an  $\m$-primary ideal. Let $\eF = \{ \A_n \}_{n \geq 0}$  be an  asymptotic  integral closure filtration \wrt \ $\A$ with $e_3^\eF(A) = 0$.  Set $L^\eF(A) = \bigoplus_{n \geq 0}A/\A_{n+1}$ considered as a  module over the Rees algebra $\ta = A[\A t]$. Let $\M = \m\oplus \ta_+$ be the maximal homogeneous ideal of $\ta$. Then the local cohomology modules $H^i_\M(L^\eF(A))$ vanish for $i = 1, 2$
\end{lemma}
We will also need to extend Lemma \ref{crucial-itoh} to dimensions $d \geq 4$.
\s \emph{Remark and a Convention:} Note that all the relevant graded modules considered upto Lemma \ref{rachel-itoh}  below are modules over the Rees algebra $\ta = A[\A t]$.  Also all local cohomology will
taken over $\M = \m\oplus \ta_+$  the maximal homogeneous ideal of $\Sc$. Note $G_\A(A)$ is a quotient of $\ta$.
We also note that if $x$ is $A$-superficial \wrt \ $\eF$ then the Rees algebra $\ta^\prime  = A/(x)[\A/(x)t]$ is a quotient of $\ta$. As we are only interested in vanishing of certain local-cohomology modules, by the independence theorem of local cohomology
it does not matter if we take local cohomology of an $\ta^\prime$-module \wrt \ $\M^\prime$ or over $\M$ (and considering the module in question as an $\ta$-module). So throughout we will only write $H^i(-)$ to mean $H^i_\M(-)$.

\begin{lemma}\label{crucial-itoh-2}
Let $(A,\m)$ be an analytically unramified  \CM \    local ring of dimension $d \geq 3$ and let
 $\A$ be an  $\m$-primary ideal. Let $\eF = \{ \A_n \}_{n \geq 0}$  be an  asymptotic  integral closure filtration \wrt \ $\A$ with $e_3^\eF(A) = 0$.  Set $L^\eF(A) = \bigoplus_{n \geq 0}A/\A_{n+1}$ considered as a  module over the Rees algebra $\ta = A[\A t]$. Let $\M = \m\oplus \Sc_+$ be the maximal homogeneous ideal of $\ta$. Then the local cohomology modules $H^i_\M(L^\eF(A))$ vanish for $i = 1, \ldots, d-1$.
\end{lemma}
\begin{proof}
This is similar to proof of Theorem \cite[7.7]{Pu-nor} and therefore omitted.
\end{proof}

We now give a proof of Theorem \ref{itoh} assuming Lemma \ref{crucial-itoh}.
\begin{proof}[Proof of Theorem \ref{itoh}]
by Lemma \ref{crucial-itoh-2} we get $H^i(L^\eF(A)) = 0$  for $1 \leq i \leq d -1$. Also as $\eF$ is the integral closure filtration by \ref{int-cls} we get $H^0(L^\A(A)) = 0$. By taking cohomology of the first fundamental sequence \ref{dag} we get that $H^i(G_\A(A)^*) = 0$
for  $0 \leq i \leq d -1$. Thus $G_\A(A)^*$ is \CM.
\end{proof}

\s \label{red-itoh} Thus to prove Itoh's conjecture all we have to do is to prove Lemma \ref{crucial-itoh}. This requires several preparatory results.
\emph{For the rest of this section we will assume $\dim A = 3$ and $\eF = \{ \A_n \}$ is an asymptotic integral closure filtration \wrt \ $\A$  such that $e_3^{\eF}(A) = 0$}. We will also assume that the residue field of $A$ is infinite.
We first show the following.
\begin{lemma}\label{asymp-itoh}
Under the hypotheses as in \ref{red-itoh},  $G_{\eF^{<l>}}(A)$ is \CM \ for $n \gg 0$
and  the $a$-invariant of $G_\eF(A)$ is negative. Furthermore $G_{\eF}(A)$ is generalized \CM.
\end{lemma}
\begin{proof}[Proof of Lemma \ref{asymp-itoh}]
Some Veronese of $\ta$ is standard graded. Say $\ta^{<l>} = A[Jt]$. Then $J$ is an asymptotically normal ideal with
 $e_3^J(A) = 0$. Then by  \cite[7.9, 7.11]{Pu-nor},  $G_{J^n}(A)$ is \CM \ for $n \gg 0$ and the $a$-invariant of $G_J(A)$ is negative. The result follows from \ref{asymp-m}, \ref{Hoa-m} and \ref{gcm-N}.
\end{proof}
Next we show
\begin{lemma}\label{rachel-itoh}
Set $G = G_\eF(A)$ and $L = L^\eF(A)$.
Under the hypotheses as in \ref{red-itoh}, we have
\begin{enumerate}[ \rm(1)]
\item
$H^2(L) = 0$.
\item
$\sum_{i = 0}^{2}(-1)^i\ell(H^i(G_\A(A)))  = 0.$
\item
For $i = 0, 1, 2$ we have $H^i(G)_n = 0$ for $n < 0$. Furthermore $H^2(G)_0 = 0$.
\end{enumerate}
\end{lemma}
\begin{proof}
 (1)
 The first fundamental exact sequence $0 \rt G \rt L \rt L(-1) \rt 0$ yields an exact sequence in cohomology
 \[
 H^2(L)_n \rt H^2(L)_{n-1} \rt H^3(G)_n.
 \]
As $a(G) < 0$ we get for $n \geq 0$, surjections $H^2(L)_{n} \rt H^2(L)_{n-1}$. As $H^2(L)_n = 0$ for $n \gg 0$ we get
$H^2(L)_n = 0$ for $n \geq -1$.

After passing through a general extension we may choose $x$, an $A$-superficial element
\wrt  \ $\A$ such that in the ring $B = A/(x)$ the quotient filtration $\eF$ is asymptotic integral closure filtration with respect to $\A/(x)$. Also notice $\dim B = 2$.
Set $\ov{L} = L^{\ov{\eF}}(B)$. By \ref{depth2}  we get that $H^1(\ov{L})_n = 0$ for $n < 0$.
By \ref{longH} we have an exact sequence
\[
 H^1(\ov{L})_n \rt H^2(L)_{n-1} \rt H^2(L)_n
\]
By setting $n = -1$ we get $H^2(L)_{-2} = 0$. Iterating we get $H^2(L)_n = 0$ for all $n \leq -2$.
It follows that $H^2(L) = 0$.

(2) As $G$ is generalized \CM \ we get that $H^i(G)$ has finite length for $i = 0, 1, 2$.
We also have $H^0(L)$
and $H^1(L)$ have finite length (see \ref{zero-lc}, \ref{depth2} and \ref{hh-m}).

The first fundamental exact sequence $0 \rt G \rt L \rt L(-1) \rt 0$ yields an exact sequence in cohomology
 \begin{align*}
  0 &\rt H^0(G) \rt H^0(L) \rt H^0(L)(-1) \\
  &\rt H^1(G) \rt H^1(L) \rt H^1(L)(-1) \\
  &\rt H^2(G) \rt H^2(L) = 0.
 \end{align*}
 Taking lengths, the result follows.

 (3) As $\eF$ is a asymptotic integral closure  filtration \wrt \ $\A$ we have that $H^1(L)_n = 0$ for $n < 0$ (see \ref{depth2} and \ref{hh-m}). Also by \ref{zero-lc} we get $H^0(L)_n = 0$
 for $n < 0$. The result follows from the above exact sequence in cohomology.
 \end{proof}

\begin{remark}\label{apply-gor-approx}
 Till now we have not used our theory of complete intersection approximation. We do it now.
 We first complete $A$. Let $\A^* = (a_1,\ldots, a_n)$ (minimally). If $\mu(\A^*) = 3$ then  $a_1, a_2, a_3$ is an $A$-regular sequence and by a result of Goto, see \cite[1.1]{G}, we get that $(\A^*)^n$ is integrally closed for all $n \geq 1$.
 It follows that $\A_n = (\A^*)^n$ for all $n \geq 1$.
    So $ G_\eF(A)$ is \CM \ and thus we have nothing to prove. Therefore we assume  $\mu(\A^*) \geq 4$. We take a general extension $A' = A[X_1,\ldots, X_n]_{\m A[X_1,\ldots, X_n]}$. Let $y = \sum_{i=1}^{n}a_iX_i$.
 Then the quotient filtration $\ov{\eF}$ is a asymptotic integral closure filtration \wrt \ $\A/(x)$. Set $B = A'/(y)$. By \ref{hh-m} that there exists $c$ such that $G_{\ov{\eF}^{<n>}}(B)$ is \CM \ for $n \geq c$.  We take a CI-approximation $(R,\n,\B,\psi)$ of $A$ \wrt \ $\A^*$ (note \emph{not of} $A'$). By \ref{ci-approx} we may assume $\red(\B, R) \geq c+2$. By construction of CI-approximations, \cite[5.3]{Pu-nor} we may assume  that $\B$ is generated by $b_1,\ldots,b_n$ and $\psi(b_i) = a_i$ for all $i$. Now set $R' = R[X_1,\ldots, X_n]_{\n R[X_1,\ldots, X_n]}$. By \ref{lying-above} we get that $A' = A\otimes_R R'$. We note that $(R',\n',\B',\psi')$ is a CI-approximation of $(A',\m', (\A^*)')$. Set $z = \sum_{i=1}^{n}b_iX_i$. Then note $\psi'(z) = y$.
Also note that $z$ is $R'$-superficial \wrt \ $\B'$, see \cite[2.6]{Ciu} . We now complete $R'$ \wrt \ $\n'$.
 Thus we may assume that our ring $(A,\m, \A^*)$ has a CI-approximation $[T,\tf,\q,\psi]$ such that
 \begin{enumerate}
 \item Both $A$ and $T$ are complete.
   \item reduction number of $\q$ is $\geq c + 2$.
   \item there exists $z \in \q$ which is $\eF$-superficial such that \\ $G_{\ov{\eF}^{<n>}}(A/zA)$ is \CM \ for all $n \geq c$. Furthermore $z^*$ is $G_\q(T)$-regular.
   \item $A$ has a canonical module $\omega_A = \Hom_T(A, T)$.
   \item By \ref{dual} there exists an $\q$-stable filtration $\eG$  on $\omega_A$ \wrt \ $\q$,  such that we have an isomorphism
   \[
   \R(\eG, \omega_A) \cong \Hom_{\R_q(T)}(\R(\eF, A), \R_\q(T)).
   \]
   As $\q \omega_A = \A^* \omega_A$, it follows that $\eG$ is infact an $\A$-stable filtration on $\omega_A$.
   \end{enumerate}

   Let $\R_\q(T)$ be the extended Rees-algebra of $T$ \wrt \ $\q$.  Let $\N$ be the $*$-maximal ideal of $\R_\q(T)$.
\end{remark}
We first show the following.
\begin{lemma}\label{gcm}
 Under the hypotheses as in \ref{apply-gor-approx}, we have:
 \begin{enumerate}[\rm (1)]
  \item For any prime $P$ in $\R_\q(T)$ with $\htt P \leq 3$ we have $\R(\eF, A)_P$ is \CM.
  \item
  $H^i_\N(\R(\eF, A))$ has finite length for $0 \leq i \leq 3$.
 \end{enumerate}
\end{lemma}
\begin{proof}
This is similar to proof of Lemma 7.13 in \cite{Pu-nor} and so is omitted.
\end{proof}
As a consequence we get
\begin{corollary}\label{exact-loc}
Under the hypotheses as in \ref{apply-gor-approx}, we have
\begin{enumerate}[\rm (1)]
  \item $H^3_\N(\R(\eF, A)) = 0$
  \item an exact sequence
\[
 0 \rt H^3_\N(G_\eF(A)) \rt H^4_\N(\R(\eF,A))(+1) \xrightarrow{t^{-1}} H^4_\N(\R(\eF, A)) \rt 0.
\]
  \item $\Ext^1_{\R_\q(T)}(\R(\eF, A),\R_\q(T)) = 0$.
\end{enumerate}
\end{corollary}
\begin{proof}
This is similar to proof of Corollary 7.14 in \cite{Pu-nor} and so is omitted.
\end{proof}
The following result is a crucial ingredient in proving Lemma \ref{crucial-itoh}.
\begin{theorem}\label{ing}Under the hypotheses as in \ref{apply-gor-approx},
 let $E_G$ be the injective hull of $k$ considered as a
 $G_\q(T)$-module. Set $W = \Hom_{G_\q(T)}(H^3(G_\eF(A)), E_G)$. Let $s = \red(\q, T)$. Let $\eG$ be the dual filtration on $\omega_A$ with respect to $\eF$. Set $\eH = \eG(s-2)$. Then
 \begin{enumerate}[\rm (1)]
  \item
  $G_\eH(\omega_A) \cong W[1]$
 \item
 $ \depth G_\eH(\omega_A) \geq 2$. Furthermore $z^*$ is $G_\eH(\omega_A)$-regular.
 \item
 $\eH_n = \omega_A$ for $n \leq 0$.
 \item
 Set $\ov{T} = T/zT$, $\ov{\q} = \q/(z)$ and $B = A/zA$. Let $\ov{\eF}$ be the quotient filtration of $\eF$ and let $\ov{\eG}$ be quotient filtration of $\eG$.
 Let $\eG^\prime$ be the dual filtration on $\omega_B$ \wrt \ $\ov{\eF}$. Then $\ov{\eG} = \eG^\prime$.
 \item
 $(\R(\ov{\eH}(1),\omega_B))^{<s-1>}$ is a \CM \ $\R_{\q^{s-1}}(\ov{T})$-module.
 \item
 $\R(\ov{\eH}^{<s-1>}, \omega_B)$ is a \CM \ $\R_{\q^{s-1}}(\ov{T})$-module.
\item
 $H^1(G_{\ov{\eH}}(\omega_B))_n = 0$ for $n < 0$.
 \end{enumerate}
\end{theorem}
\begin{proof}
(1) We note that the $a$-invariant of $G_\q(T)$ is $s-3$. So the $a$-invariant of $\R_\q(T)$ is $s-2$.
Let $E_\R$ be the injective hull of $k$ considered as a $\R_\q(T)$-module.  Set $(-)^\vee = \Hom_{\R_q(T)}(-, E_\R)$.
Dualizing the exact sequence in Lemma \ref{exact-loc} we get a sequence of $\R_\q(T)$-modules
\[
 0 \rt H^4_\N(\R(\eF, A))^\vee \xrightarrow{t^{-1}}  H^4_\N(\R(\eF, A))^\vee(-1) \rt W \rt 0.
\]
As $R_\q(T)$ is Gorenstein then by local duality we have an exact sequence
$$ 0 \rt V \xrightarrow{t^{-1}} V(-1) \rt W \rt 0, $$
where
$$ V = \Hom_{\R_q(T)}(\R(\eF, A), \R_\q(T)(s-2)). $$
It follows that $G_\eH(\omega) = W[1]$

(2)
We note that  $W \cong \Hom_{G_\q(T)}(G_\eF(A), G_\q(T) (s-3))$. We have that $G_\q(T)$ is \CM \ of dimension three and $z^*$ is $G_\q(T)$-regular.  The result follows.

(3)
 Set $G = G_\eF(A)$.
Let $E = H^3(G)$. By \ref{rachel-itoh}(1) $a(G) < 0$. So  $E_n = 0$ for $n\geq 0$. Thus $W_m = 0$
for $m \leq 0$. Therefore $G_\eH(\omega_A)_n = 0$ for $n \leq -1$. The result follows.

(4) We have an exact sequence
\[
0 \rt \R_\q(T)(-1) \xrightarrow{zt} \R_\q(T) \rt   \R_\q(\ov{T}) \rt 0.
\]
This yields an exact sequence
\begin{align*}
  0 \rt \Hom_{\R_\q(T)}(\R(\eF, A), \R_\q(T))(-1) &\xrightarrow{zt} \Hom_{\R_\q(T)}(\R(\eF, A), \R_\q(T)) \rt  \\
  \Hom_{\R_\q(T)}(\R(\eF, A), \R_\q(\ov{T})) &\rt \Ext^1_{\R_\q(T)}(\R(\eF, A), \R_\q(T))(-1) \rt.
\end{align*}
By \ref{exact-loc} we get that $\Ext^1_{\R_\q(T)}(\R(\eF, A), \R_\q(T)) = 0$.
We note that
\[
\Hom_{\R_\q(T)}(\R(\eF, A), \R_\q(\ov{T})) = \Hom_{\R_\q(T)}(\R(\eF, A)/zt \R(\eF, A), \R_\q(\ov{T})).
\]
 We have an exact sequence
\[
0 \rt E = \bigoplus_{n \geq 1}\frac{\A_n \cap (z) }{ z\A_{n-1}} \rt \R(\eF, A)/zt \R(\eF, A) \rt \R(\ov{\eF}, B) \rt 0.
\]
As $z$ is $A$-superficial \wrt \ $\eF$  we get that $E$ has finite length. As $\R_\q(\ov{T})$ is \CM \  and has dimension $3$ we get $\Hom_{\R_\q(T)}(E, \R_\q(\ov{T})) = 0$. So
\[
\Hom_{\R_\q(T)}(\R(\eF, A)/zt \R(\eF, A) , \R_\q(\ov{T})) \cong \Hom_{\R_\q(T)}(\R(\ov{\eF}, B), \R_\q(\ov{T})).
\]
Thus we have an exact sequence
\[
0 \rt \R(\eG, \omega_A)(-1)\xrightarrow{zt} \R(\eG,\omega_A) \rt \R(\eG^\prime,\omega_B) \rt 0;
\]
where $\eG^\prime$ is the dual filtration on $\omega_B$ with respect to $\ov{\eG}$.
As $z^*$ is $G_\eG(\omega_A)$-regular it follows that $\ov{\eG} = \eG^\prime$.

(5) By our construction $\R(\ov{\eF}^{<n>}, B)$ is \CM \ for $n \geq c$. Also $s = \red(\q, T) = \red(\ov{\q}, \ov{T}) \geq c + 2$.
As $G_{\ov{\q}}(\ov{T})$ is Gorenstein we get by \cite[Theorem 1]{Ooishi-gor-powers}  that $G_{\ov{\q}^{s-1}}(\ov{T})$ is Gorenstein. It follows that $Q = \R_{\ov{\q}^{s-1}}(\ov{T})$ is Gorenstein. Therefore $\Hom_{Q}(\R(\ov{\eF}^{<s-1>}, B), Q))$ is a \CM \ $Q$-module.
By \ref{dual-Ver} we have an isomorphism of $Q$-modules
\[
\Hom_{\R(\ov{\q}, T)}(\R(\ov{\eF}, B),   \R_{\ov{\q}}( \ov{T} ))^{<s-1>}  \cong \Hom_{Q}(\R(\ov{\eF}^{<s-1>}, B),Q))
\]
Thus $\R(\ov{\eG}, \omega_B)^{<s-1>}$ is a \CM \ $Q$-module.
Set
\[
U = \R(\ov{\eG}(s-1), \omega_B)^{<s-1>} \quad  \text{and} \quad V =  \left(\R(\ov{\eG}, \omega_B)^{<s-1>}\right)(s-1).
\]
We claim $U = V$.\\
Proof:
\[
U = (\bigoplus_{n \in \Z}\ov{\eG}_{n + s-1})^{<s-1>} = \bigoplus_{n\in \Z}\ov{\eG}_{(n+ s-1)(s-1)},
\]
while
\[
V_n = (\R(\ov{\eG}, \omega_B))^{<s-1>}_{n+s-1} = \ov{\eG}_{(n+s-1)(s-1)}.
\]
Thus the assertion made in the Claim holds.

It follows that $ \R(\ov{\eG}(s-1), \omega_B)^{<s-1>}$ is a \CM \ $Q$-module.
So \\ $\R(\ov{\eH}(1), \omega_B)^{<s-1>}$ is a \CM \ $Q$-module.

(6) We have $\R(\ov{\eH}, \omega_B) = \bigoplus_{n \in \Z} \ov{\eH}_n$. So we get
\[
\R(\ov{\eH}(1), \omega_B)^{<s-1>} = \bigoplus_{n \in \Z} \ov{\eH}_{(n+1)(s-1)}.
\]
We also have $\R(\ov{\eH}, \omega_B)^{<s-1>} = \bigoplus_{n \in \Z} \ov{\eH}_{n(s-1)}$. Thus we have
\[
\left(\R(\ov{\eH}, \omega_B)^{<s-1>}\right)_n = \ov{\eH}_{n(s-1)} = \left(\R(\ov{\eH}(1), \omega_B)^{<s-1>}\right)_{n-1}.
\]
So we have an equality of $Q = \R_{\ov{\q^{s-1}}}(\ov{T})$-modules
\[
\left(\R(\ov{\eH}, \omega_B)^{<s-1>}\right) \cong \left(\R(\ov{\eH}(1), \omega_B)^{<s-1>}\right)(-1).
\]
Thus $\R(\ov{\eH}, \omega_B)^{<s-1>}$ is a \CM \ $Q$-module.  We also note  that \\ $\R(\ov{\eH}, \omega_B)^{<s-1>} = \R(\ov{\eH}^{<s-1>}, \omega_B)$.

(7) By (6) we get that $G_{\ov{\eH}^{<s-1>}}(\omega_B)$ is \CM. The result now follows from \ref{depth2}.
\end{proof}

We now give Proof of Lemma \ref{crucial-itoh}.
\begin{proof}
By  Lemma \ref{rachel-itoh}(2) we get $H^2(L^\eF(A)) = 0$. By \ref{dag} we have an exact sequence
\[
H^1(L^\eF(A)) \rt H^1(L^\eF(A))(-1) \rt H^2(G_\eF(A)) \rt H^2(L^\eF(A)) = 0
\]
As $H^1(L^\eF(A))_n = 0$ for $n < 0$ it follows that $H^2(G_\eF(A))_n = 0$ for $n \leq 0$. Note that to prove $ H^1(L^\eF(A)) = 0$ it suffices to show $H^2(G_\eF(A)) = 0$ (because \\
$H^1(L^\eF(A))$ has finite length).
Set $G = G_\eF(A)$ and $\ov{G} = G/ztG$. We have a exact sequence
\[
0 \rt K \rt G(-1) \xrightarrow{zt} G \rt \ov{G} \rt 0;
\]
where $K$ has finite length. Taking local cohomology we get an exact sequence
\begin{align*}
  H^2(G)(-1)&\xrightarrow{zt} H^2(G) \rt H^2(\ov{G}) \rt \\
  H^3(G)(-1)&\xrightarrow{zt} H^3(G) \rt 0.
\end{align*}
Set $Y = H^2(G)/ztH^2(G)$. Note $Y_n = 0$ for $n \leq 0$. Taking Matlis dual's  we get an exact sequence
\[
0 \rt H^3(G)^\vee \xrightarrow{zt} H^3(G)^\vee(1) \rt H^2(\ov{G})^\vee \rt Y^\vee \rt 0.
\]
By \ref{ing} we get an exact sequence
\[
0 \rt G_{\ov{\eH}}(\omega_B)  \rt H^2(\ov{G})^\vee \rt Y^\vee \rt 0.
\]
We note that $\depth H^2(\ov{G})^\vee \geq 2$ and $H^1(G_{\ov{\eH}} (\omega_B))_n = 0$ for $n < 0$.  Also $Y$ has finite length. The long exact sequence in cohomology implies  $Y^\vee_n = 0$ for $n < 0$.
It follows that $Y_n = 0$ for $n > 0$. As $Y_n = 0$ for $n \leq 0$ we get $Y = 0$. So $H^2(G)/zt H^2(G) = 0$. By graded Nakayama's lemma we get $H^2(G) = 0$. As discussed before this implies that $H^1(L^\eF(A)) = 0$.
\end{proof}

\end{document}